\documentclass[11pt]{amsart}

\usepackage{amssymb,amsmath,amsthm,newlfont}
\usepackage[shortlabels]{enumitem}
\usepackage{xcolor}

\usepackage[colorlinks=true, citecolor=red, linkcolor=blue]{hyperref}

\theoremstyle{plain}
\newtheorem{theorem}{Theorem}[section]
\newtheorem{proposition}[theorem]{Proposition}
\newtheorem{corollary}[theorem]{Corollary}
\newtheorem{lemma}[theorem]{Lemma}

\theoremstyle{definition}

\newtheorem{example}[theorem]{Example}

\theoremstyle{remark}
\newtheorem*{remark}{Remark}

\renewcommand{\hat}{\widehat}

\newcommand{\DD}{\mathbb{D}}

\newcommand{\TT}{\mathbb{T}}
\newcommand{\cD}{\mathcal{D}}
\newcommand{\cH}{\mathcal{H}}

\begin{document}

\date{9 May 2023}

\title[Divergence of Taylor series]{On the divergence of Taylor series \\ in de Branges--Rovnyak spaces}

\author{Pierre-Olivier Paris\'e}
\address{Department of Mathematics, University of Hawaii at Manoa, Honolulu, HI 96822, USA.}
\email{parisepo@hawaii.edu}

\author{Thomas Ransford}
\address{D\'epartement de math\'ematiques et de statistique, Universit\'e Laval,
Qu\'ebec City (Qu\'ebec), G1V 0A6, Canada.}
\email{ransford@mat.ulaval.ca}

\thanks{POP supported by a grant from NSERC and FQRNT.
TR supported by grants from NSERC and the Canada Research Chairs program.}

\begin{abstract}
It is known that there exist functions in certain de Branges--Rovnyak spaces
whose Taylor series diverge in norm,
even though polynomials are dense in the space.
This is often proved by showing that the
sequence of Taylor partial sums is unbounded in norm.
In this note we show that it can even happen that the Taylor partial sums  tend to infinity in norm. 
We also establish similar results for
lower-triangular summability methods such as the Ces\`aro means.
\end{abstract}

\makeatletter
\@namedef{subjclassname@2020}{\textup{2020} Mathematics Subject Classification}
\makeatother



\maketitle

\section{Introduction}\label{S:intro}

The Hilbert function spaces now known as de Branges--Rovnyak spaces
were first introduced by de Branges and Rovnyak in the appendix of \cite{dBR66}, in relation to scattering models. Later on, thanks largely to Sarason's work culminating in \cite{Sa94}, these spaces turned out to have deep connections with operator theory. 
For a recent account of the theory of de Branges--Rovnyak spaces,
see the two-volume monograph \cite{FM16a, FM16b}.

Let $H^\infty$ denote the space of bounded holomorphic functions on the open unit disk $\DD$.
Given $b$ in the unit ball of $H^\infty$,
the \emph{de Branges--Rovnyak space} $\cH (b)$
is the (unique) reproducing kernel Hilbert space on $\DD$
with kernel $k^b(z, w) = (1 - \overline{b(w)} b(z))/(1 - \overline{w} z)$, for $w, z \in \DD$.
It is a Hilbert space  of holomorphic functions on $\DD$, contractively contained in the Hardy space $H^2$.

The properties enjoyed by $\cH(b)$ depend strongly on whether 
$b$ is or is not an extreme point of the unit ball of $H^\infty$.
In particular, $\cH(b)$ contains the polynomials if and only if $b$ is non-extreme, and
in this case the polynomials are dense in $\cH(b)$ (see \cite[IV-3]{Sa97}).
However, even if the polynomials are dense in $\cH(b)$,
the Taylor series of  functions in $\cH(b)$ may diverge in norm.
In this note, we examine this phenomenon in detail.

It was shown in \cite[Theorem~5.5]{CGR10} that there exists a 
non-extreme function $b$ in the unit ball of $H^\infty$
and a function $f\in\cH(b)$ whose radial dilates satisfy
\begin{equation}\label{E:limsupfr}
\limsup_{r\to1^-}\|f_r\|_{\cH(b)}=\infty .
\end{equation}
Somewhat later, in \cite[Theorem~3.1]{EFKMR16}, 
an example was given
of a non-extreme  $b$ and an $f\in\cH(b)$
with the stronger property that
\begin{equation}\label{E:liminffr}
\liminf_{r\to1^-}\|f_r\|_{\cH(b)}=\infty.
\end{equation}
Another construction of such an example can be found 
in \cite[Theorem~1.1]{MR18}.

As was pointed out in \cite[Corollary~3.4]{EFKMR16}, the relation
\eqref{E:limsupfr} implies that
\begin{equation}\label{E:limsupsn}
\limsup_{n\to\infty}\|s_n(f)\|_{\cH(b)}=\infty
\quad\text{and}\quad
\limsup_{n\to\infty}\|\sigma_n(f)\|_{\cH(b)}=\infty,
\end{equation}
where $s_n(f)$ is the sequence of Taylor partial sums of $f$
and $\sigma_n(f)$ is the sequence of their Ces\`aro means.
However, the corresponding stronger form of \eqref{E:limsupsn}, namely
\begin{equation}\label{E:liminfsn}
\liminf_{n\to\infty}\|s_n(f)\|_{\cH(b)}=\infty
\quad\text{and}\quad
\liminf_{n\to\infty}\|\sigma_n(f)\|_{\cH(b)}=\infty,
\end{equation}
does not immediately follow from \eqref{E:liminffr}.
It raises the question as to whether there exist examples
where \eqref{E:liminfsn} holds. 
Our aim in this note is to
present an affirmative answer to this question. 
This is not a trivial endeavour since,
as we shall see, there exist de Branges--Rovnyak spaces $\cH(b)$
in which $\limsup_{n\to\infty}\|s_n(f)\|_{\cH(b)}=\infty$ for  some
$f$, yet $\liminf_{n\to\infty}\|s_n(f)\|_{\cH(b)}<\infty$
for all $f$.


\section{Statement of main result}

We shall  treat the case of a general lower-triangular
summability method.
Let $\Gamma=(\gamma_{nk})_{n,k\ge0}$ be an infinite
lower-triangular matrix of complex numbers. Given a formal
power series $f(z)=\sum_{k\ge0}\hat{f}(k)z^k$, for each $n\ge0$ we define
\[
S^{\Gamma}_n(f)(z):=\sum_{k=0}^n \gamma_{nk}\hat{f}(k)z^k.
\]

The following theorem is our main result.

\begin{theorem}\label{T:main}
Let $b$ be a non-extreme point of the unit ball of $H^\infty$,
and let $\Gamma=(\gamma_{nk})_{n,k\ge0}$ be an infinite
lower-triangular matrix of complex numbers.
\begin{enumerate}[\normalfont(i)]
\item If $\sup_{n\ge0}(|\gamma_{nn}|\|z^n\|_{\cH(b)})=\infty$, then there exists $f\in\cH(b)$ such that
\[
\limsup_{n\to\infty}\|S_n^\Gamma(f)\|_{\cH(b)}=\infty.
\]
\item If $\sum_{n\ge0}(|\gamma_{nn}|^{-2}\|z^n\|_{\cH(b)}^{-2})<\infty$,
 then there exists $f\in\cH(b)$ such that
\[
\liminf_{n\to\infty}\|S_n^\Gamma(f)\|_{\cH(b)}=\infty.
\]
\item If $\sum_{n\ge0}(|\gamma_{nn}|^{-1}\|z^n\|_{\cH(b)}^{-1})<\infty$,
 then there exist $f,g\in\cH(b)$ such that
\[
\liminf_{n\to\infty}|\langle S_n^\Gamma(f),g\rangle_{\cH(b)}|=\infty.
\]
\end{enumerate}
\end{theorem}

An important special case is obtained by taking 
$\gamma_{nk}=\binom{n}{k}/\binom{n+\alpha}{k}$ for $0\le k\le n$,
where $\alpha\ge0$.
Then $S^\Gamma_n(f)$ is just the generalized
Ces\`aro mean $\sigma_n^\alpha(f)$. 
In particular  $S^\Gamma_n(f)=s_n(f)$ if $\alpha=0$
and $S^\Gamma_n(f)=\sigma_n(f)$ if $\alpha=1$.
Since 
\[
|\gamma_{nn}|=\binom{n+\alpha}{n}^{-1}\asymp n^{-\alpha}
\quad(n\to\infty),
\]
we immediately obtain the following corollary.

\begin{corollary}\label{C:main}
Let $b$ be a non-extreme point of the unit ball of $H^\infty$,
and let $\alpha\ge0$.
\begin{enumerate}[\normalfont(i)]
\item If $\sup_{n\ge0}(n^{-\alpha}\|z^n\|_{\cH(b)})=\infty$, then there exists $f\in\cH(b)$ such that
\[
\limsup_{n\to\infty}\|\sigma_n^\alpha(f)\|_{\cH(b)}=\infty.
\]
\item If $\sum_{n\ge0} (n^{2\alpha}/\|z^n\|_{\cH(b)}^{2})<\infty$,
 then there exists $f\in\cH(b)$ such that
\[
\liminf_{n\to\infty}\|\sigma_n^\alpha(f)\|_{\cH(b)}=\infty.
\]
\item If $\sum_{n\ge0} (n^{\alpha}/\|z^n\|_{\cH(b)})<\infty$,
 then there exist $f,g\in\cH(b)$ such that
\[
\liminf_{n\to\infty}|\langle \sigma_n^\alpha(f),g\rangle_{\cH(b)}|=\infty.
\]
\end{enumerate}
\end{corollary}

\begin{remark}
Part~(i) of Corollary~\ref{C:main} was already known.
It was established in \cite[Theorem~6.10]{GMR23}.
Our proof here is quite different,
and also leads to parts~(ii) and (iii), which we believe to be new,
even in the special cases of $s_n(f)$ and $\sigma_n(f)$.
\end{remark}


\section{Proof of main result}\label{S:proof}

The proof is based on the following uniform boundedness principles from functional analysis.

\begin{theorem}\label{T:ubp}
Let $(T_n)$ be a sequence of bounded operators on a
complex Hilbert space~$H$.
\begin{enumerate}[\normalfont(i)]
\item If $\sup_n\|T_n\|=\infty$, then there exists $x\in H$ such that
\[
\limsup_{n\to\infty}\|T_nx\|=\infty.
\]
\item If $\sum_n 1/\|T_n\|^2<\infty$, then
there exists $x\in H$ such that
\[
\liminf_{n\to\infty}\|T_nx\|=\infty.
\]
\item If $\sum_n 1/\|T_n\|<\infty$, then
there exist $x,y\in H$ such that
\[
\liminf_{n\to\infty}|\langle T_nx,y\rangle|=\infty.
\]
\end{enumerate}
\end{theorem}

\begin{proof}
Part~(i) is just the standard uniform bounded principle.
Parts~(ii) and (iii) are proved in \cite[Theorems~3\,(ii) and 6\,(ii)]{MV09}, where they are derived 
as a consequence of Ball's solution to the complex plank problem
\cite{Ba01}.
\end{proof}

We are going to apply these results with $T_n=S^\Gamma_n$
and $H=\cH(b)$.
Each $S^\Gamma_n$ is a bounded linear operator on
$\cH(b)$. The following
lemma provides a lower bound for its operator norm.
Once this lemma is established, Theorem~\ref{T:main} will follow immediately.

Notice that $b$ is  a non-extreme point of the unit ball of $H^\infty$
if and only if there exists an outer function $a\in H^\infty$ with $a(0)>0$
such that $|b|^2+|a|^2=1$ a.e.\ on $\TT$.
The function $a$ is then uniquely determined. 
For more details, see \cite[Chapter~IV]{Sa94}.
We call $(b,a)$ a \emph{Pythagorean pair}. 

\begin{lemma}\label{L:operatornorm}
Let $(b,a)$ be a  Pythagorean pair, 
and let $\Gamma=(\gamma_{nk})$
be a lower-triangular matrix of complex numbers. Then,
for all $n\ge0$,
\[
\|S^\Gamma_n:\cH(b)\to\cH(b)\|\ge a(0)|\gamma_{nn}|\|z^n\|_{\cH(b)}.
\]
\end{lemma}

\begin{proof}
It is  known that, if $h\in H^2$, then $ah\in \cH(b)$ and $\|ah\|_{\cH(b)}\le \|h\|_{H^2}$
(see \cite[Lemma~4]{Sa86}). In particular, setting $g_n(z):=z^na(z)$, we have that $g_n\in\cH(b)$ and
$\|g_n\|_{\cH(b)}\le1$.
Also $g_n(z)=a(0)z^n+(\text{terms of higher degree})$, so $S^\Gamma_n(g_n)=\gamma_{nn}a(0)z^n$.
Combining these observations, we deduce that
\[
\|S^\Gamma_n:\cH(b)\to\cH(b)\|
\ge \frac{\|S^\Gamma_n(g_n)\|_{\cH(b)}}{\|g_n\|_{\cH(b)}}
\ge \frac{|\gamma_{nn}a(0)|\|z^n\|_{\cH(b)}}{1}.
\]
This proves the lemma.
\end{proof}

\section{Examples}\label{S:examples}

In order to exploit the results of the preceding section,
we need to be able to calculate, or at least estimate, the norms of monomials
$\|z^n\|_{\cH(b)}$. There is a well-known formula
that enables us to do this.
The following result was established in \cite[p.81]{Sa86}.
 
\begin{proposition}\label{P:formula}
Let $(b,a)$ be a Pythagorean pair and let $\phi:=b/a$,
say $\phi(z)=\sum_{j=0}^\infty c_jz^j$. Then
\[
\|z^n\|_{\cH(b)}^2=1+\sum_{j=0}^n|c_j|^2 \quad(n\ge0).
\]
\end{proposition}

Clearly, if $(b,a)$ is a Pythagorean pair, then $\phi:=b/a\in N^+$,
the Smirnov class. Conversely, given $\phi\in N^+$,
then there exists  a unique Pythagorean pair $(b,a)$ 
such that $\phi=b/a$ (see \cite[Proposition~3.1]{Sa08}). 
Thus we can identify $\cH(b)$ by specifying the corresponding function $\phi\in N^+$, and this is what we shall do in the examples below.
 
\begin{example}\label{X:localD}
Let  $\phi(z):=z/(\zeta-z)$, where $\zeta\in\TT$.

The Taylor coefficients $c_j$ of $\phi$ satisfy
$c_0=0$ and $|c_j|=1$ for $j\ge1$, so by Proposition~\ref{P:formula}
we have $\|z^n\|_{\cH(b)}= \sqrt{n+1}$.
Corollary~\ref{C:main} implies that there exists
$f\in\cH(b)$ such that $\limsup_{n\to\infty}\|s_n(f)\|_{\cH(b)}=\infty$.
However, this result does not permit us to deduce the existence of an $f\in\cH(b)$ with $\liminf_{n\to\infty}\|s_n(f)\|_{\cH(b)}=\infty$
or $\limsup_{n\to\infty}\|\sigma_n(f)\|_{\cH(b)}=\infty$.
In fact, we claim that no $f\in\cH(b)$ satisfies either of these
conclusions,
indicating the Corollary~\ref{C:main} is rather sharp, at least for this function~$\phi$.

To justify the claim, we note that,
for this particular $\phi$, 
it is known that $b(z)=(1-\tau) z/(\zeta-\tau z)$,
where $\tau=(3-\sqrt{5})/2$, and that $\cH(b)$ is equal
to the local Dirichlet space $\cD_\zeta$, with equality of norms
(see \cite[Proposition~2]{Sa97}). 
It was proved in \cite[Theorem~2.4]{MSW22} that
\begin{equation}\label{E:snfinite}
\liminf_{n\to\infty}\|s_n(f)\|_{\cD_\zeta}=\|f\|_{\cD_\zeta}
\quad(f\in\cD_\zeta).
\end{equation}
Also, it was shown in \cite[Theorem~1.6]{MR19} that 
$\|\sigma_n(f)- f\|_{\cD_\zeta}\to0$
for all $f\in\cD_\zeta$, so in particular
\begin{equation}\label{E:sigmanfinite}
\limsup_{n\to\infty}\|\sigma_n(f)\|_{\cD_\zeta}=\|f\|_{\cD_\zeta}
\quad(f\in\cD_\zeta).
\end{equation}
Since $\cD_\zeta=\cH(b)$, it follows that both
\eqref{E:snfinite} and \eqref{E:sigmanfinite} hold
with $\cD_\zeta$ replaced everywhere by $\cH(b)$, confirming the
claim made above.
\end{example}

\begin{example}
Let $\phi(z)=z^M/(1-z)^N$, where $M,N$ are integers
with $M\ge0$ and $N\ge1$.

In this case $|c_j|\asymp j^{N-1}$, so by Proposition~\ref{P:formula}
we have $\|z^n\|_{\cH(b)}\asymp n^{N-1/2}$ as $n\to\infty$.
Applying Corollary~\ref{C:main}, we deduce that, if 
$0\le \alpha<N-1$, then there exists $f\in\cH(b)$ such that
$\|\sigma_n^\alpha(f)\|_{\cH(b)}\to\infty$ as $n\to\infty$.
In particular, if $N\ge 2$, then there exists
$f\in\cH(b)$ with $\|s_n(f)\|_{\cH(b)}\to\infty$.
Likewise, if $N\ge 3$, then there exists
$f\in\cH(b)$ with $\|\sigma_n(f)\|_{\cH(b)}\to\infty$.
This answers the question posed in the introduction.
\end{example}

\begin{example}
Let $\phi(z)=\exp(\beta/(1-z)^\gamma)$,
where $\beta>0$ and $\gamma\in(0,1)$.

In this case, 
it can be shown that 
\[
c_j\asymp  \frac{\exp(Cj^{\gamma/(\gamma+1)})}{j^{(\gamma+2)/(2\gamma+2)}} \quad(j\to\infty),
\]
where $C$ is a positive constant depending on $\beta,\gamma$
(see Appendix~\ref{S:appendix} below).
By Proposition~\ref{P:formula}, it follows that
\[
\|z^n\|_{\cH(b)}^2\gtrsim
\frac{\exp(Cn^{\gamma/(\gamma+1)})}{n^{(\gamma+2)/(2\gamma+2)}} \quad(n\to\infty).
\]
By Corollary~\ref{C:main}, 
for each $\alpha\ge0$, there exist functions $f,g\in\cH(b)$
such that 
$|\langle\sigma_n^\alpha(f),g\rangle_{\cH(b)}|\to\infty$. 
In particular,  $\|\sigma_n^\alpha(f)\|_{\cH(b)}\to\infty$.

This raises the following question: does there exist
a function  $f\in\cH(b)$
such that  $\|\sigma_n^\alpha(f)\|_{\cH(b)}\to\infty$
for all $\alpha\ge0$?
\end{example}


\appendix

\section{The Taylor coefficients of $\exp(\beta/(1-z)^\gamma))$}\label{S:appendix}

Let $\phi(z):=\exp(\beta/(1-z)^\gamma)$, where $\beta>0$ and $\gamma\in(0,1)$.
This function belongs to the Smirnov class, since it is the exponential of a function in 
the Hardy space $H^1$.
To estimate its Taylor coefficients, 
we  use  the following generalization of Stirling's formula
due to Hayman \cite[p.69]{Ha56}.

\begin{theorem}\label{T:Hayman}
Let $f(z)=\sum_{n=0}^\infty c_nz^n$ be an admissible holomorphic function on $\DD$
such that $f(r)>0$ for all $r\in(0,1)$. For $r\in(0,1)$, define
\begin{align*}
M(r)&:=\sup_{|z|=r}|f(z)|,\\
A(r)&:=r(\log M(r))',\\
B(r)&:=rA'(r).
\end{align*}
Then, as $n\to\infty$,
\begin{equation}\label{E:asymp}
c_n\sim \frac{f(r_n)}{r_n^n\sqrt{2\pi B(r_n)}},
\end{equation}
where $r_n\in(0,1)$ is the unique solution to the equation $A(r_n)=n$.
\end{theorem}

The definition of admissible function is complicated, and we do not give it here.
However, the function $\phi(z)$ defined above is admissible (see \cite[p.93]{Ha56}),
so Theorem~\ref{T:Hayman} does indeed apply to it. Here is what we deduce.

\begin{theorem}
Let $\phi(z)=\exp(\beta/(1-z)^\gamma)$, where $\beta>0$ and $\gamma\in(0,1)$.
Then $\phi(z)=\sum_{n=0}^\infty c_nz^n$, where
\begin{equation}\label{E:cn}
c_n\sim  \frac{\exp(Cn^{\gamma/(\gamma+1)})}{Dn^{(\gamma+2)/(2\gamma+2)}} \quad(n\to\infty),
\end{equation}
and where 
\begin{equation}\label{E:constants}
C=(\beta\gamma)^{1/(\gamma+1)}(1+1/\gamma)
\quad\text{and}\quad 
D= \sqrt{\frac{2\pi(\gamma+1)}{(\beta\gamma)^{1/(\gamma+1)}}}.
\end{equation}
\end{theorem}

\begin{proof}
We apply Theorem~\ref{T:Hayman}. Simple computations give
\[
M(r) =\exp\Bigl(\frac{\beta}{(1-r)^\gamma}\Bigr),
\quad
A(r) = \frac{\beta\gamma r}{(1-r)^{\gamma+1}},
\quad
B(r) =\frac{\beta\gamma r(1+\gamma r)}{(1-r)^{\gamma+2}}.
\]
Also, the relation $A(r_n)=n$ is equivalent to
$\beta\gamma r_n/n=(1-r_n)^{\gamma+1}$.
Writing $s_n:=1-r_n$, we obtain
\[
s_n=\Bigl(\frac{\beta\gamma}{n}\Bigr)^{1/(\gamma+1)}(1-s_n)^{1/(\gamma+1)}.
\]
In particular
$s_n=O(n^{-1/(\gamma+1)})$ as $n\to\infty$, and so
\[
s_n=\Bigl(\frac{\beta\gamma}{n}\Bigr)^{1/(\gamma+1)}\Bigl(1+O(n^{-1/(\gamma+1)})\Bigr).
\]

A computation gives that
\begin{align*}
\log \phi(r_n)&=\frac{\beta}{(1-r_n)^\gamma}=\frac{\beta}{s_n^\gamma}
=\beta\Bigl(\frac{n}{\beta\gamma}\Bigr)^{\gamma/(\gamma+1)}\Bigl(1+O(n^{-1/(\gamma+1)})\Bigr)\\
&=\beta\Bigl(\frac{n}{\beta\gamma}\Bigr)^{\gamma/(\gamma+1)}+o(1),
\end{align*}
so
\begin{equation}\label{E:phi(r)}
\phi(r_n)\sim\exp\Bigl(\frac{\beta^{1/(\gamma+1)}}{\gamma^{\gamma/(\gamma+1)}}n^{\gamma/(\gamma+1)}\Bigr).
\end{equation}
Also, 
\begin{align*}
\log(r_n^n)&=n\log(1-s_n)=-ns_n+O(ns_n^2)\\
&=-n\Bigl(\frac{\beta\gamma}{n}\Bigr)^{1/(\gamma+1)}\Bigl(1+O(n^{-1/(\gamma+1)})\Bigr)+O(n^{(\gamma-1)/(\gamma+1)}\Bigr)\\
&=-(\beta\gamma)^{1/(\gamma+1)}n^{\gamma/(\gamma+1)}+o(1),
\end{align*}
so
\begin{equation}\label{E:rn^n}
r_n^n\sim\exp\Bigl(-(\beta\gamma)^{1/(\gamma+1)}n^{\gamma/(\gamma+1)}\Bigr).
\end{equation}
Lastly, we have
\begin{equation}\label{E:b(rn)}
B(r_n)=\frac{\beta\gamma r_n(1+\gamma r_n)}{(1-r_n)^{\gamma+2}}
\sim\frac{\beta\gamma(1+\gamma)}{s_n^{\gamma+2}}
\sim\frac{1+\gamma}{(\beta\gamma)^{1/(\gamma+1)}}n^{(\gamma+2)/(\gamma+1)}.
\end{equation}
Feeding \eqref{E:phi(r)}, \eqref{E:rn^n} and \eqref{E:b(rn)} into \eqref{E:asymp},
we obtain \eqref{E:cn} and \eqref{E:constants}.
\end{proof}

\bibliographystyle{amsplain}
\bibliography{bibliography.bib}
\end{document}